\documentclass{article}

\usepackage{amsmath,amssymb}
\usepackage{tikz}
\usepackage{color}
\usepackage[toc]{appendix}
\usepackage{graphicx}
\usepackage{fancyhdr}
\usepackage{enumitem}
\usepackage{bbm}
\usepackage{graphicx}
\usepackage{parskip}
\usepackage{float}
\usepackage{chngpage}
\usepackage{calc}
\usepackage{bigints}
\usepackage{tikz}
\usepackage{array}
\usepackage{bbm}
\usepackage{booktabs}
\usepackage{rotating}
\usepackage{multirow}
\usepackage{adjustbox}
\usepackage{tabularx}
\usepackage{enumitem}
\usepackage{verbatim}
\usepackage{mathtools}
\usepackage{ragged2e}
\usepackage[makeroom]{cancel}
\usepackage{enumitem}
\usepackage{caption}
\usepackage{mathtools}
\usepackage{hyperref}
\usepackage{bbm}
\usepackage{amsmath,amsfonts,amsthm,bm}
\usepackage{amsmath,amssymb}
\usepackage{appendix}
\usepackage{graphicx}
\usepackage{pgfplots}
\usepackage{tikz}
\usepackage{verbatim}
\usepackage{amsthm}
\usepackage[english]{babel}
\usepackage{hyphenat}
\usepackage[makeindex]{imakeidx}
\usepackage{tikz}

\usetikzlibrary{datavisualization}
\usetikzlibrary{matrix}
\usetikzlibrary{datavisualization.formats.functions}

\newtheorem{theorem}{Theorem}[section]

\newtheorem{lemma}[theorem]{Lemma}

\newtheorem{remark}[theorem]{Remark}
\newtheorem{assumption}[theorem]{Assumption}

\makeatletter
\def\section{\@startsection {section}{1}{\z@}{3.25ex plus 1ex minus
		.2ex}{1.5ex plus .2ex}{\large\bf}}
\def\subsection{\@startsection{subsection}{2}{\z@}{3.25ex plus 1ex minus
		.2ex}{1.5ex plus .2ex}{\normalsize\bf}}
\@addtoreset{equation}{section} 
\makeatother

\title{A small time approximation for the solution to the Zakai Equation}

\author{Alberto Lanconelli\thanks{Dipartimento di Scienze Statistiche Paolo Fortunati, Università di Bologna, Bologna, Italy. \textbf{e-mail}: alberto.lanconelli2@unibo.it} \and  Ramiro Scorolli\thanks{Dipartimento di Scienze Statistiche Paolo Fortunati, Università di Bologna, Bologna, Italy. \textbf{e-mail}: ramiro.scorolli2@unibo.it}}

\date{\today}

\begin{document}
	
	\maketitle
	
	\bigskip
	
	\begin{abstract}
	We propose a novel small time approximation for the solution to the Zakai equation from nonlinear filtering theory. We prove that the unnormalized filtering density is well described over short time intervals by the solution of a deterministic partial differential equation of Kolmogorov type; the observation process appears in a pathwise manner through the degenerate component of the Kolmogorov's type operator. The rate of convergence of the approximation is of order one in the lenght of the interval. Our approach combines ideas from Wong-Zakai-type results and Wiener chaos approximations for the solution to the Zakai equation. The proof of our main theorem relies on the well-known Feynman-Kac representation for the unnormalized filtering density and careful estimates which lead to completely explicit bounds.     	
	\end{abstract}
	
	Key words and phrases: nonlinear filtering, Zakai equation, Feynmn-Kac formula, Wick product. \\
	
	AMS 2000 classification: 60G35, 60H15, 60H07.
	
	\bigskip
	
	\allowdisplaybreaks
	
\section{Introduction and statement of the main result}\label{introduction}
	
In this short note we derive a new small time approximation for the solution to the Zakai equation
\begin{align}\label{Zakai SPDE}
u(t,x)=u_0(x)+\int_0^t\mathcal{L}_x^{\star}u(s,x)ds+\int_0^th(x)u(s,x)dY_s,\quad t\in [0,1],x\in\mathbb{R}^d.
\end{align}	
Here:
\begin{itemize}
\item $\mathcal{L}_x^{\star}$ is the formal adjoint of $\mathcal{L}_x$, generator of the $d$-dimensional \emph{signal} process $\{X_t\}_{t\in [0,1]}$ which is assumed to solve the stochastic differential equation
\begin{align}\label{signal}
X_t=X_0+\int_0^tb(X_s)ds+\int_0^t\sigma(X_s)dB_s,\quad t\in [0,1];
\end{align}
the process $\{B_t\}_{t\in [0,1]}$ is a standard $d$-dimensional Brownian motion defined on the complete probability space $(\Omega,\mathcal{F},\mathbb{P})$; 
\item  $\{Y_t\}_{t\in [0,1]}$ is the one dimensional \emph{observation} process described by
\begin{align}\label{observation}
Y_t=y_0+\int_0^th(X_s)ds+W_t,\quad t\in [0,1],
\end{align}
with $\{W_t\}_{t\in [0,1]}$ being a standard one-dimensional Brownian motion defined on $(\Omega,\mathcal{F},\mathbb{P})$ and independent of $\{B_t\}_{t\geq 0}$. 
\end{itemize}
The solution $\{u(t,x)\}_{t\in [0,1],x\in\mathbb{R}^d}$ to the Zakai equation \eqref{Zakai SPDE}, usually called \emph{unnormalized filtering density}, plays a crucial role in the nonlinear filtering problem since it identifies uniquely the conditional distribution of $X_t$ given $\mathcal{F}^Y_t:=\sigma(Y_s, 0\leq s\leq t)$. The reader is referred to the original paper \cite{Zakai} and the references quoted there; for an exhaustive treatment of the subject we suggest the excellent review \cite{Heunis}, as well as the books \cite{Kallianpur book} and \cite{Lipster}. \\
Existence, uniqueness and regularity properties for the solution to (\ref{Zakai SPDE}) can be found for instance, under different sets of assumptions and solution concepts, in the classic works \cite{Budhiraja},\cite{Krylov},\cite{Kunita},\cite{Pardoux} and the more recent paper \cite{BP}. We also mention a useful Feynman-Kac representation for the solution $\{u(t,x)\}_{t\in [0,1],x\in\mathbb{R}^d}$ obtained in \cite{Kunita} and, in a slightly different form, in \cite{BP}. This representation will play a crucial role in our investigation. \\
From the applications point of view, closed form expressions for the solution to the Zakai equation are certainly desirable; however, as pointed in \cite{Benes} only few particular cases of \eqref{Zakai SPDE} allow for explicit computations. The important issue of deriving simple approximation schemes for the solution to (\ref{Zakai SPDE}) have been considered in \cite{Bensoussan} and \cite{De Masi} which employ splitting up methods and time discretization, respectively; Wong-Zakai-type results were investigated in \cite{Chaleyat} and \cite{HKX} while \cite{Budhiraja} and \cite{Lotosky} proposed a Wiener chaos approach. We also mention the so called \emph{pathwise filtering} that steams from the problem of having a robust, with respect to the observation process, filter; this has been discussed in \cite{Clark} and \cite{Davis}.  \\
The approach proposed in the current paper combines ideas from the Wong-Zakai approximation proposed in \cite{HKX}, where the signal process is smoothed through a polygonal approximation, and the Wiener chaos approach presented in \cite{Budhiraja} and \cite{Lotosky}, where one relates equation (\ref{Zakai SPDE}) to a system of nested deterministic partial differential equations solved by the kernels of the Cameron-Martin decomposition of the solution $\{u(t,x)\}_{t\in [0,1],x\in\mathbb{R}^d}$. We refer the reader to Remark \ref{heuristic} below for the heuristic idea supporting our analysis and its link to the aforementioned approaches. \\
The main novelty of our result is the connection between equation (\ref{Zakai SPDE}) and a deterministic partial differential equation of Kolmogorov type (see e.g. \cite{OR}), where the observation process enters as a degenerate component of the second order differential operator $\mathcal{L}_x^{\star}$. We prove that the solution $\{u(t,x)\}_{t\in [0,1],x\in\mathbb{R}^d}$ to the Zakai equation \eqref{Zakai SPDE} can be approximated over small intervals of time by the solution of the aforementioned degenerate partial differential equation, with the observation process having a pathwise role. This approximation has the same rate of convergence of one obtained in \cite{Lotosky} and is described by completely explicit constants.\\
To be more specific, we now introduce some notation and state our main result. In the sequel the following regularity conditions will be in force.

\begin{assumption}\label{assumption 1}\quad
	\begin{enumerate}	
	\item For $1\leq i,j\leq d$, the functions $b_i:\mathbb{R}^d\to\mathbb{R}^d$ and $a_{ij}:\mathbb{R}^d\to\mathbb{R}^d$, where
	\begin{align}\label{aij}
	a_{ij}(x):=\sum_{k=1}^d\sigma_{ik}(x)\sigma_{jk}(x),\quad x\in\mathbb{R}^d,
	\end{align}
	are bounded with bounded partial derivatives up to the third order. Moreover, the matrix $\{a_{ij}(x)\}_{1\leq i,j\leq d}$ is uniformly elliptic, i.e. there exists two positive constants $\mu_1<\mu_2$ such that
	\begin{align*}
	\mu_1|z|^2\leq \sum_{i,j=1}^da_{ij}(x)z_iz_j\leq \mu_2|z|^2,\quad\mbox{ for all $z\in\mathbb{R}^d$}, 
	\end{align*}
	with $|z|^2:=z_1^2+\cdot\cdot\cdot+z_d^2$.
	\item The initial data $X_0$ in \eqref{signal} is random, independent of $\{B_t\}_{t\in [0,1]}$ and its distribution is absolutely continuous with respect to the $d$-dimensional Lebesgue measure; its density $u_0:\mathbb{R}^d\to\mathbb{R}$ is bounded and acts as initial data in (\ref{Zakai SPDE}). 
    \item The function $h:\mathbb{R}^d\to\mathbb{R}$ is bounded and globally Lipschitz continuous.
    \end{enumerate}
\end{assumption}

\begin{remark}\label{contants}
We observe that, according to Assumption \ref{assumption 1}, there exists a positive constant $L$ such that
\begin{align}\label{L}
	|h(x_1)-h(x_2)|\leq L|x_1-x_2|,\quad\mbox{ for all $x_1,x_2\in\mathbb{R}^d$}.	
\end{align}	
Moreover, there exists a positive constant $M$ such that 
\begin{align}\label{M}
\max\{|a(x)|^2,|b^{\star}(x)|\}\leq M(1+|x^2|),\quad\mbox{ for all $x\in\mathbb{R}^d$},	
\end{align} 
where $b^{\star}_i(x):=\sum_{j=1}^d\partial_{x_j}a_{ij}(x)-b_i(x)$, $i=1,...,d,$. We will need these two constants in the statement of our main theorem.
\end{remark}

According to the Girsanov theorem and thanks to the assumption of boundedness on $h$, the prescription
\begin{align*}
\mathbb{P}_1(A):=\int_{A}e^{-\int_0^1h(X_s(\omega))dW_s(\omega)-\frac{1}{2}\int_0^1h(X_s(\omega))^2ds}d\mathbb{P}(\omega),\quad A\in\mathcal{F},
\end{align*}
defines a probability measure on $(\Omega,\mathcal{F})$; moreover, the stochastic process $\{Y_t-y_0\}_{t\in [0,1]}$ in (\ref{observation}) becomes on the probability space $(\Omega,\mathcal{F},\mathbb{P}_1)$ a one dimensional Brownian motion independent of $\{B_t\}_{t\geq 0}$.
In the sequel we will write $\mathbb{E}_1$ to denote the expectation under the probability measure $\mathbb{P}_1$.

We are now ready to state our main result.

\begin{theorem}\label{main theorem}
	Let Assumption \ref{assumption 1} be in force and, for $0<T<1$, let 
	\begin{align*}
		[0,T]\times\mathbb{R}^d\times\mathbb{R}\ni(t,x,y)\mapsto v(t,x,y)
	\end{align*}
	be a classical solution of the Cauchy problem
	\begin{align}
		\begin{cases}\label{Kolmogorov operator}
			\partial_t v(t,x,y)=\mathcal{L}_x^{\star}v(t,x,y)-h(x)\partial_yv(t,x,y),\quad (t,x,y)\in]0,T]\times\mathbb{R}^d\times\mathbb{R};\\
			v(0,x,y)=u_0(x)e^{-\frac{y^2}{2T}},\quad (x,y)\in\mathbb{R}^d\times\mathbb{R}.
			\end{cases}
		\end{align}	
Then, for any $q\geq 1$ and $K>0$, we have
\begin{align}\label{estimate}
\sup_{|x|\leq K}\mathbb{E}_1\left[\left|u(T,x)-e^{\frac{(Y_{T}-y_0)^2}{2T}}v(T,x,Y_{T}-y_0)\right|^q\right]^{1/q}\leq \mathcal{C} T,
\end{align}
with 
\begin{align}\label{contant}
\mathcal{C}:=\frac{2}{\sqrt{3}}|u_0|_{\infty}e^{T\left(|c|_{\infty}+\frac{q_1-1}{2}|h|_{\infty}^2+\sqrt{M}+M/2\right)}\left(\kappa(q_2)+\sqrt{T}|h|_{\infty}\right)L\sqrt{2(1+K^2)(1+T)}.
\end{align}
Here $L$ and $M$ are defined in \eqref{L} and \eqref{M}, respectively; the constants $q_1,q_2\geq 1$ verify the identity $\frac{1}{q_1}+\frac{1}{q_2}=\frac{1}{q}$; $\kappa(q_2)$ is given by $\sqrt{2}\left(\Gamma(\frac{q_2+1}{2})/\sqrt{\pi}\right)^{1/q_2}$; $|u_0|_{\infty}$ and $|h|_{\infty}$ denotes the $L^{\infty}(\mathbb{R}^d)$-norms of $u_0$ and $h$, respectively. 
\end{theorem}

\begin{remark}\label{heuristic}
The heuristic idea that links equation \eqref{Zakai SPDE} to equation \eqref{Kolmogorov operator} is as follows. Write \eqref{Zakai SPDE} in the differential form 	
\begin{align}\label{diff Zakai SPDE}
	\partial_tu(t,x)=\mathcal{L}_x^{\star}u(t,x)+h(x)u(t,x)\diamond\frac{dY_t}{dt},\quad u(0,x)=u_0(x),
\end{align}		
where $\diamond$ denotes the Wick product associated to the Brownian motion $\{Y_t-y_0\}_{t\in [0,1]}$ on the probability space $(\Omega,\mathcal{F},\mathbb{P}_1)$. The use of the Wick product is dictated  by the It\^o's interpretation of \eqref{Zakai SPDE} (see \cite{HOUZ} and \cite{Janson} for a discussion on this issue and detailed analysis of the Wick product). If equation \eqref{diff Zakai SPDE} is considered on a small time interval $[0,T]$, one may replace $\frac{dY_t}{dt}$ with $\frac{Y_T-y_0}{T}$ (this amounts at considering a Wong-Zakai approximation with the rudest possible partition of the interval $[0,T]$); this gives
\begin{align}\label{diff Zakai SPDE 2}
	\partial_tu(t,x)=\mathcal{L}_x^{\star}u(t,x)+\frac{h(x)}{T}u(t,x)\diamond (Y_T-y_0),\quad u(0,x)=u_0(x).
\end{align}
In general, the Wick-multiplication between a random variable $X$ and an element from the first order Wiener chaos, say $I(f)$, can be rewritten as
\begin{align*}
X\diamond I(f)=X\cdot I(f)-D_fX,
\end{align*}
where $D_fX$ stands for the directional Malliavin derivative of $X$, in the direction $\int_0^{\cdot}f(s)ds$ (see \cite{Nualart}). Since, $Y_T-y_0=\int_0^1\boldsymbol{1}_{[0,T]}(s)dY_s$ is an element in the first Wiener chaos associated with the Brownian motion $\{Y_t-y_0\}_{t\in [0,1]}$ and probability space $(\Omega,\mathcal{F},\mathbb{P}_1)$, we can transform equation \eqref{diff Zakai SPDE 2} into
\begin{align}\label{Zakai increment}
\partial_tu(t,x)=\mathcal{L}_x^{\star}u(t,x)+\frac{h(x)}{T}u(t,x) (Y_T-y_0)-\frac{h(x)}{T}D_{\boldsymbol{1}_{[0,T]}}u(t,x).	
\end{align}
We now search for a solution $u(t,x)$ to equation \eqref{Zakai increment} of the form
\begin{align}\label{guess}
	u(t,x,\omega)=\tilde{u}(t,x,Y_T(\omega)-y_0),
\end{align}
for some $\tilde{u}:[0,T]\times\mathbb{R}^d\times\mathbb{R}\to\mathbb{R}$ to be determined. A substitution of \eqref{guess} in \eqref{Zakai increment} yields, together with the chain rule for the Malliavin derivative, 
\begin{align*}
	\partial_t\tilde{u}(t,x,Y_T-y_0)=&\mathcal{L}_x^{\star}\tilde{u}(t,x,Y_T-y_0)+\frac{h(x)}{T}\tilde{u}(t,x,Y_T-y_0)(Y_T-y_0)\\
	&-h(x)\partial_y\tilde{u}(t,x,Y_T-y_0);	
\end{align*}
note that here the term $Y_T-y_0$ can be tackled at a path-wise level. Equation \eqref{Kolmogorov operator} is now obtained via the simple transformation
\begin{align*}
v(t,x,y):=\tilde{u}(t,x,y)e^{-\frac{y^2}{2T}}, \quad t \in [0,T],x\in\mathbb{R}^d,y\in\mathbb{R}.
\end{align*}
It is not difficult to see, using Theorem 4.12 in \cite{Janson} and the Feynman-Kac representation for $\{u(t,x)\}_{t\in [0,1],x\in\mathbb{R}^d}$ in \cite{BP}, that we also have
\begin{align*}
\mathbb{E}_1[u(T,x)|Y_{T}-y_0]=e^{\frac{(Y_{T}-y_0)^2}{2T}}v(T,x,Y_{T}-y_0);
\end{align*}
this spots the analogy between our approach and the one in \cite{Lotosky} where projections of $u(T,x)$ on suitable families of elements from the Wiener chaos were utilized to propose approximation schemes for the solution to (\ref{Zakai SPDE}).    
\end{remark}

\section{Proof of Theorem \ref{main theorem}}\label{section proof main theorem}

We start with some notation. The generator $\mathcal{L}_x$ of the signal process $\{X_t\}_{t\in [0,1]}$ in (\ref{signal}) is
\begin{align*}
\mathcal{L}_xf(x)=\frac{1}{2}\sum_{i,j=1}^da_{ij}(x)\partial^2_{x_ix_j}f(x)+\sum_{i=1}^db_i(x)\partial_{x_i}f(x),
\end{align*}
where the $a_{ij}(x)$'s are defined in (\ref{aij}). The adjoint operator $\mathcal{L}^{\star}_x$ is given by
\begin{align*}
\mathcal{L}^{\star}_xf(x)=\frac{1}{2}\sum_{i,j=1}^da_{ij}(x)\partial^2_{x_ix_j}f(x)+\sum_{i=1}^db^{\star}_i(x)\partial_{x_i}f(x)+c(x)f(x),
\end{align*}
with
\begin{align*}
b^{\star}_i(x):=\sum_{j=1}^d\partial_{x_j}a_{ij}(x)-b_i(x),\quad i=1,...,d,
\end{align*}
(see Remark \ref{contants}) and
\begin{align*}
c(x):=\sum_{i=1}^d\left(\frac{1}{2}\sum_{j,k=1}^d\partial^2_{x_jx_k}a_{ij}(x)-\partial_{x_k}b_i(x)\right).
\end{align*}
It is convenient to split the operator $\mathcal{L}^{\star}_x$ as
\begin{align*}
\mathcal{L}^{\star}_xf(x)=\mathtt{L}^{\star}_xf(x)+c(x)f(x)
\end{align*}
where we set
\begin{align*}
\mathtt{L}^{\star}_xf(x):=\frac{1}{2}\sum_{i,j=1}^da_{ij}(x)\partial^2_{x_ix_j}f(x)+\sum_{i=1}^db^{\star}_i(x)\partial_{x_i}f(x).
\end{align*}
With this notation at hand, the Cauchy problem (\ref{Kolmogorov operator}) takes the form
\begin{align}\label{PDE}
\begin{cases}
\partial_t v(t,x,y)=\mathtt{L}^{\star}_xv(t,x,y)+c(x)v(t,x,y)-h(x)\partial_yv(t,x,y)\\ \quad (t,x,y)\in]0,T]\times\mathbb{R}^d\times\mathbb{R};\\
v(0,x,y)=u_0(x)e^{-\frac{y^2}{2T}},\quad (x,y)\in\mathbb{R}^d\times\mathbb{R}.
\end{cases}
\end{align}

Now, assume 
\begin{align*}
[0,T]\times\mathbb{R}^d\times\mathbb{R}\ni(t,x,y)\mapsto v(t,x,y)
\end{align*}
to be a classical solution of (\ref{PDE}). According to the Feynman-Kac formula (see, for instance, Theorem 1.1, page 120, and the comments at page 122 in \cite{Freidlin}), we can write
\begin{align*}
v(T,x,y)&=\hat{\mathbb{E}}\left[u_0(\hat{\xi}_{T}^x)e^{-\frac{\left(y-\int_0^{T}h(\hat{\xi}_s^x)ds\right)^2}{2T}}e^{\int_0^{T}c(\hat{\xi}_s^x)ds}\right]\\
&=e^{-\frac{y^2}{2T}}\hat{\mathbb{E}}\left[u_0(\hat{\xi}_{T}^x)e^{\int_0^{T}c(\hat{\xi}_s^x)ds}e^{\frac{y\int_0^{T}h(\hat{\xi}_s^x)ds}{T}-\frac{\left(\int_0^{T}h(\hat{\xi}_s^x)ds\right)^2}{2T}}\right],
\end{align*}
where $\{\hat{\xi}_s^x\}_{s\in [0,1]}$ solves the SDE
\begin{align}\label{xi}
d\hat{\xi}_s^x=b^{\star}(\hat{\xi}_s^x)+\sigma(\hat{\xi}_s^x)d\hat{B}_s,\quad \hat{\xi}_0^x=x,
\end{align}
on the auxiliary probability space $(\hat{\Omega},\hat{\mathcal{F}},\hat{\mathbb{P}})$ with $d$-dimensional Brownian motion $\{\hat{B}_s\}_{s\in [0,1]}$. This gives
\begin{align}\label{FK v}
e^{\frac{(Y_T-y_0)^2}{2T}}v(T,x,Y_T-y_0)=\hat{\mathbb{E}}\left[u_0(\hat{\xi}_T^x)e^{\int_0^Tc(\hat{\xi}_s^x)ds}e^{\frac{(Y_T-y_0)\int_0^Th(\hat{\xi}_s^x)ds}{T}-\frac{\left(\int_0^Th(\hat{\xi}_s^x)ds\right)^2}{2T}}\right].
\end{align}

It is well known that the solution $u(t,x)$ to the Zakai equation (\ref{Zakai SPDE}) also possesses a Feynman-Kac representation: see formula (1.4) page 132 in \cite{Kunita}. Here, we use instead an equivalent formulation due to \cite{BP} (see formula (2.9) there), namely
\begin{align}\label{FK u}
u(T,x)=\hat{\mathbb{E}}\left[u_0(\hat{\xi}_T^x)e^{\int_0^Tc(\hat{\xi}_s^x)ds}e^{\int_0^Th(\hat{\xi}_{T-s}^x)dY_s-\frac{1}{2}\int_0^Th^2(\hat{\xi}_s^x)ds}\right],
\end{align}
where $\{\hat{\xi}_s^x\}_{s\in [0,1]}$ is defined in (\ref{xi}). A comparison between (\ref{FK v}) and (\ref{FK u}) gives
\begin{align*}
&u(T,x)-e^{\frac{(Y_T-y_0)^2}{2T}}v(T,x,Y_T-y_0)\\
&\quad=\hat{\mathbb{E}}\left[u_0(\hat{\xi}_T^x)e^{\int_0^Tc(\hat{\xi}_s^x)ds}e^{\int_0^Th(\hat{\xi}_{T-s}^x)dY_s-\frac{1}{2}\int_0^Th^2(\hat{\xi}_s^x)ds}\right]\\
&\quad\quad-\hat{\mathbb{E}}\left[u_0(\hat{\xi}_T^x)e^{\int_0^Tc(\hat{\xi}_s^x)ds}e^{\frac{(Y_T-y_0)\int_0^Th(\hat{\xi}_s^x)ds}{T}-\frac{\left(\int_0^Th(\hat{\xi}_s^x)ds\right)^2}{2T}}\right]\\
&\quad=\hat{\mathbb{E}}\left[u_0(\hat{\xi}_T^x)e^{\int_0^Tc(\hat{\xi}_s^x)ds}\left(e^{\int_0^Th(\hat{\xi}_{T-s}^x)dY_s-\frac{1}{2}\int_0^Th^2(\hat{\xi}_s^x)ds}-e^{\frac{(Y_T-y_0)\int_0^Th(\hat{\xi}_s^x)ds}{T}-\frac{\left(\int_0^Th(\hat{\xi}_s^x)ds\right)^2}{2T}}\right)\right],
\end{align*}
and hence
\begin{align*}
&\left|u(T,x)-e^{\frac{(Y_T-y_0)^2}{2T}}v(T,x,Y_T-y_0)\right|\\
&\quad\leq\hat{\mathbb{E}}\left[|u_0(\hat{\xi}_T^x)|e^{\int_0^Tc(\hat{\xi}_s^x)ds}\left|e^{\int_0^Th(\hat{\xi}_{T-s}^x)dY_s-\frac{1}{2}\int_0^Th^2(\hat{\xi}_s^x)ds}-e^{\frac{(Y_T-y_0)\int_0^Th(\hat{\xi}_s^x)ds}{T}-\frac{\left(\int_0^Th(\hat{\xi}_s^x)ds\right)^2}{2T}}\right|\right]\\
&\quad\leq | u_0|_{\infty}e^{T|c|_{\infty}}\hat{\mathbb{E}}\left[\left|e^{\int_0^Th(\hat{\xi}_{T-s}^x)dY_s-\frac{1}{2}\int_0^Th^2(\hat{\xi}_s^x)ds}-e^{\frac{(Y_T-y_0)\int_0^Th(\hat{\xi}_s^x)ds}{T}-\frac{\left(\int_0^Th(\hat{\xi}_s^x)ds\right)^2}{2T}}\right|\right].
\end{align*}
We now take $q\geq 1$ and compute the $L^q(\mathbb{P}_1)$-norm of the first and last members above; an application of Minkowsky's inequality gives
\begin{align}\label{b}
&\left\Vert u(T,x)-e^{\frac{(Y_T-y_0)^2}{2T}}v(T,x,Y_T-y_0)\right\Vert_q\nonumber\\
&\quad\leq | u_0|_{\infty}e^{T|c|_{\infty}}\hat{\mathbb{E}}\left[\left\Vert e^{\int_0^Th(\hat{\xi}_{T-s}^x)dY_s-\frac{1}{2}\int_0^Th^2(\hat{\xi}_s^x)ds}-e^{\frac{(Y_T-y_0)\int_0^Th(\hat{\xi}_s^x)ds}{T}-\frac{\left(\int_0^Th(\hat{\xi}_s^x)ds\right)^2}{2T}}\right\Vert_q\right].
\end{align} 

We need the following result.

\begin{lemma}\label{lemma}
Let $f,g:[0,T]\to\mathbb{R}$ be bounded measurable deterministic functions. Then, for any $q\geq 1$ we have
\begin{align*}
&\mathbb{E}_1\left[\left|e^{\int_0^Tf(s)dY_s-\frac{1}{2}|f|^2_{2}}-e^{\int_0^Tg(s)dY_s-\frac{1}{2}|g|^2_{2}}\right|^q\right]^{\frac{1}{q}}\\
&\quad\leq\left(e^{\frac{q_1-1}{2}T|f|_{\infty}^2}+e^{\frac{q_1-1}{2}T|g|_{\infty}^2}\right)\left(\kappa(q_2)+\frac{\sqrt{T}}{2}(|f|_{\infty}+|g|_{\infty})\right)|f-g|_{2}, 
\end{align*}
where $q_1,q_2\geq 1$ satisfy $1/q_1+1/q_2=1/q$ while $\kappa(q_2):=\sqrt{2}\left(\Gamma(\frac{q_2+1}{2})/\sqrt{\pi}\right)^{1/q_2}$. Moreover, $|l|_2$ and $|l|_{\infty}$ stand for the norms in $L^2([0,T])$ and $L^{\infty}([0,T])$ of $l$, respectively. 
\end{lemma}

\begin{proof}
By means of the elementary inequality $|e^a-e^b|\leq (e^a+e^b)|a-b|$ we can write
\begin{align*}
&\left|e^{\int_0^Tf(s)dY_s-\frac{1}{2}|f|^2_{2}}-e^{\int_0^Tg(s)dY_s-\frac{1}{2}|g|^2_{2}}\right|\\
&\quad\leq \left(e^{\int_0^Tf(s)dY_s-\frac{1}{2}|f|^2_{2}}+e^{\int_0^Tg(s)dY_s-\frac{1}{2}|g|^2_{2}}\right)\\
&\quad\quad\times\left|\int_0^T[f(s)-g(s)]dY_s-\frac{1}{2}\left(|f|^2_{2}-|g|^2_{2}\right)\right|\\
&\quad\leq\left(e^{\int_0^Tf(s)dY_s-\frac{1}{2}|f|^2_{2}}+e^{\int_0^Tg(s)dY_s-\frac{1}{2}|g|^2_{2}}\right)\\
&\quad\quad\times \left(\left|\int_0^T[f(s)-g(s)]dY_s\right|+\frac{1}{2}\left||f|^2_{2}-|g|^2_{2}\right|\right).
\end{align*}
Now, for $q\geq 1$ we take the $L^q(\mathbb{P}_1)$-norm of the first and last members above and apply H\"older's inequality with exponents $q_1,q_2\geq 1$ satisfying $1/q_1+1/q_2=1/q$. This gives
\begin{align}\label{aa}
&\left\Vert e^{\int_0^Tf(s)dY_s-\frac{1}{2}|f|^2_{2}}-e^{\int_0^Tg(s)dY_s-\frac{1}{2}|g|^2_{2}}\right\Vert_{q}\nonumber\\
&\quad\leq \left\Vert e^{\int_0^Tf(s)dY_s-\frac{1}{2}|f|^2_{2}}+e^{\int_0^Tg(s)dY_s-\frac{1}{2}|g|^2_{2}}\right\Vert_{q_1}\nonumber\\
&\quad\quad\times\left(\left\Vert\int_0^T[f(s)-g(s)]dY_s\right\Vert_{q_2}+\frac{1}{2}\left||f|^2_{2}-|g|^2_{2}\right|\right).
\end{align}
Under the measure $\mathbb{P}_1$, the random variables $\int_0^Tf(s)dY_s$ and $\int_0^Tg(s)dY_s$ are Gaussian with mean zero and variances $|f|^2_{2}$ and $|g|^2_{2}$, respectively. Hence,
\begin{align}\label{ab}
&\left\Vert e^{\int_0^Tf(s)dY_s-\frac{1}{2}|f|^2_{2}}+e^{\int_0^Tg(s)dY_s-\frac{1}{2}|g|^2_{2}}\right\Vert_{q_1}\nonumber\\
&\quad\leq \left\Vert e^{\int_0^Tf(s)dY_s-\frac{1}{2}|f|^2_{2}}\right\Vert_{q_1}+\left\Vert e^{\int_0^Tg(s)dY_s-\frac{1}{2}|g|^2_{2}}\right\Vert_{q_1}\nonumber\\
&\quad=e^{\frac{q_1-1}{2}|f|^2_{2}}+e^{\frac{q_1-1}{2}|g|^2_{2}}\nonumber\\
&\quad\leq e^{\frac{q_1-1}{2}T|f|_{\infty}^2}+e^{\frac{q_1-1}{2}T|g|_{\infty}^2}.
\end{align}
Moreover, using once more the normality, under the measure $\mathbb{P}_1$, of the random variable $\int_0^T[f(s)-g(s)]dY_s$ we get
\begin{align}\label{ac}
\left\Vert\int_0^T[f(s)-g(s)]dY_s\right\Vert_{q_2}&=\kappa(q_2)|f-g|_{2},
\end{align}
where $\kappa(q_2):=\sqrt{2}\left(\Gamma(\frac{q_2+1}{2})/\sqrt{\pi}\right)^{1/q_2}$ (see, for instance, Formula (1.1) in \cite{Janson}). Furthermore,
\begin{align}\label{ad}
\left||f|^2_{2}-|g|^2_{2}\right|&=\left(|f|_{2}+|g|_{2}\right)\left||f|_{2}-|g|_{2}\right|\nonumber\\
&\leq \left(|f|_{2}+|g|_{2}\right)|f-g|_{2}\nonumber\\
&\leq \sqrt{T}(|f|_{\infty}+|g|_{\infty})|f-g|_{2}.
\end{align}
Therefore, combining (\ref{aa}) with (\ref{ab}), (\ref{ac}) and (\ref{ad}) we get
\begin{align*}
&\left\Vert e^{\int_0^Tf(s)dY_s-\frac{1}{2}|f|^2_{2}}-e^{\int_0^Tg(s)dY_s-\frac{1}{2}|g|^2_{2}}\right\Vert_{q}\\
&\leq \left(e^{\frac{q_1-1}{2}T|f|_{\infty}^2}+e^{\frac{q_1-1}{2}T|g|_{\infty}^2}\right)\left(\kappa(q_2)+\frac{\sqrt{T}}{2}(|f|_{\infty}+|g|_{\infty})\right)|f-g|_{2}.
\end{align*}
The proof is complete.
\end{proof}

Thanks to the identities
\begin{align*}
\frac{(Y_T-y_0)\int_0^Th(\hat{\xi}_s^x)ds}{T}=\int_0^T\left(\frac{1}{T}\int_0^Th(\hat{\xi}_r^x)dr\right)dY_s,
\end{align*}
and 
\begin{align*}
\int_0^T\left(\frac{1}{T}\int_0^Th(\hat{\xi}_r^x)dr\right)^2ds=\frac{\left(\int_0^Th(\hat{\xi}_r^x)dr\right)^2}{T},
\end{align*}
we are in a position to apply Lemma \ref{lemma} to the last term in (\ref{b}) with
\begin{align*}
f(s):=h(\hat{\xi}_{T-s}^x)\quad\mbox{ and }\quad g(s):=\frac{1}{T}\int_0^Th(\hat{\xi}_r^x)dr;
\end{align*} 
note that such choices imply $|f|_{\infty}\leq |h|_{\infty}$ and $|g|_{\infty}\leq |h|_{\infty}$ (here, the norms are on the corresponding domains). Therefore, 
\begin{align*}
&\left\Vert u(T,x)-e^{\frac{(Y_T-y_0)^2}{2T}}v(T,x,Y_T-y_0)\right\Vert_q\nonumber\\
&\quad\leq 2|u_0|_{\infty}e^{T\left(|c|_{\infty}+\frac{q_1-1}{2}|h|_{\infty}^2\right)}\left(\kappa(q_2)+\sqrt{T}|h|_{\infty}\right)\\
&\quad\quad\times\hat{\mathbb{E}}\left[\left(\int_0^T\left|h(\hat{\xi}_{T-s}^x)-\frac{1}{T}\int_0^Th(\hat{\xi}_r^x)dr\right|^2ds\right)^{1/2}\right].
\end{align*}
We now focus on the last expectation; using a combination of Jensen's inequalities and Tonelli's theorem we get
\begin{align*}
&\hat{\mathbb{E}}\left[\left(\int_0^T\left|h(\hat{\xi}_{T-s}^x)-\frac{1}{T}\int_0^Th(\hat{\xi}_r^x)dr\right|^2ds\right)^{1/2}\right]\\
&\quad\leq \left(\hat{\mathbb{E}}\left[\int_0^T\left|h(\hat{\xi}_{T-s}^x)-\frac{1}{T}\int_0^Th(\hat{\xi}_r^x)dr\right|^2ds\right]\right)^{1/2}\\
&\quad= \left(\int_0^T\hat{\mathbb{E}}\left[\left|h(\hat{\xi}_{T-s}^x)-\frac{1}{T}\int_0^Th(\hat{\xi}_r^x)dr\right|^2\right]ds\right)^{1/2}\\
&\quad= \left(\int_0^T\hat{\mathbb{E}}\left[\left|\frac{1}{T}\int_0^T(h(\hat{\xi}_{T-s}^x)-h(\hat{\xi}_r^x))dr\right|^2\right]ds\right)^{1/2}\\
&\quad\leq \left(\int_0^T\hat{\mathbb{E}}\left[\frac{1}{T}\int_0^T|h(\hat{\xi}_{T-s}^x)-h(\hat{\xi}_r^x)|^2dr\right]ds\right)^{1/2}\\
&\quad= \left(\int_0^T\left(\frac{1}{T}\int_0^T\hat{\mathbb{E}}\left[|h(\hat{\xi}_{T-s}^x)-h(\hat{\xi}_r^x)|^2\right]dr\right)ds\right)^{1/2}.
\end{align*}
The Lipschitz continuity of $h$ and Theorem 4.3, Chapter 2 in \cite{Mao book} yield
\begin{align*}
\hat{\mathbb{E}}\left[|h(\hat{\xi}_{T-s}^x)-h(\hat{\xi}_r^x)|^2\right]&\leq L^2\hat{\mathbb{E}}\left[|\hat{\xi}_{T-s}^x-\hat{\xi}_r^x|^2\right]\\
&\leq 2L^2(1+|x|^2)(1+T)e^{2(\sqrt{M}+M/2)T}|T-s-r|,
\end{align*}
where $L$ and $M$ come from \eqref{L} and \eqref{M}. Moreover,
\begin{align*}
\left(\int_0^T\left(\frac{1}{T}\int_0^T|T-s-r|dr\right)ds\right)^{1/2}=\frac{T}{\sqrt{3}}.
\end{align*}
Combining all our estimates we obtain
\begin{align*}
&\left\Vert u(T,x)-e^{\frac{(Y_T-y_0)^2}{2T}}v(T,x,Y_T-y_0)\right\Vert_q\nonumber\\
&\quad\leq \frac{2}{\sqrt{3}}|u_0|_{\infty}e^{T\left(|c|_{\infty}+\frac{q_1-1}{2}|h|_{\infty}^2+\sqrt{M}+M/2\right)}\left(\kappa(q_2)+\sqrt{T}|h|_{\infty}\right)\\
&\quad\quad\times L\sqrt{2(1+|x|^2)(1+T)}T,
\end{align*}
as desired.

\end{document}